\documentclass[11pt]{amsart}

\usepackage{amsmath}
\usepackage[frame,ps,matrix,arrow,curve,rotate]{xy}
\usepackage{amsthm}
\usepackage{amsfonts}
\usepackage{latexsym}
\usepackage{amssymb,amscd}
\usepackage{xspace}

\newtheorem{theorem}{Theorem}[section]
\newtheorem{proposition}[theorem]{Proposition}
\newtheorem{conjecture}[theorem]{Conjecture}

\newtheorem{lemma}[theorem]{Lemma}
\newtheorem{corollary}[theorem]{Corollary}

\theoremstyle{remark}

\newtheorem{example}[theorem]{\bf Example}
\theoremstyle{definition}

\def\P{{\bf P}}

\def\A{{\bf A}}

\def\O{{\mathcal O}}
\def\cO{{\mathcal O}}
\def\cI{{\mathcal I}}

\def\Hilb{{\rm Hilb}}
\def\G{{\rm G}}

\def\deg{\mathop{\rm deg}}

\def\spec{\mathop{\rm Spec}}
\def\codim{\mathop{\rm codim}}

\def\Hom{\mathop{\rm Hom}}

\def\Hom{{\rm Hom}}

\def\Spec{\mathop{\rm Spec}}

\def\PGL{\mathop{\rm PGL}}

\def\coker{\mathop{\rm coker}}

\def\PP{{\P}}

\def\modular{semi-modular }
\def\T{{\mathcal T}}

\begin{document}

\title{Fibers of Projections and submodules of deformations}
\author{Roya Beheshti, David Eisenbud}
\date{}
\maketitle

\begin{abstract}
We  bound the complexity of the fibers of the generic linear projection
of a smooth variety in terms of a new family of invariants. These
invariants are closely related to ideas of John Mather, and we give a simple
proof of his bound on the Thom-Boardman invariants of a generic projection
as an application.
\end{abstract}

\section{Introduction}

Let $X \subset \P^r$ be a smooth projective variety of dimension $n$ over an
algebraically closed field $k$ of characteristic zero, 
and let $\pi: X \to 
\P^{n+c}$ be a general linear projection. In this note we introduce some new ways
of bounding the complexity of the fibers of $\pi$. Our ideas are closely related to 
the groundbreaking work of John Mather, and we explain a simple proof of his result
\cite{mather2} bounding the Thom-Boardman invariants of $\pi$ as a special case.

This subject was studied classically for small $n$. In our situation the map $\pi$ will
be finite and generically one-to-one, so we are asking for bounds on the complexity
of finite schemes, and the degree of the scheme is the obvious invariant.
Consider, for simplicity,
the case $c=1$. It is well-known
that the maximal degree of the fiber of a general projection of a curve
 to the plane is 2, and that the maximal degree of a fiber of a general projection
 of a smooth surface to three-space is 3. These results were extended to higher
 dimension and more general ground fields at the expense of strong hypotheses on the structure of the fibers by
 Kleiman, Roberts, Ran and others.
 
 In characteristic zero, the most striking results are those of John Mather. In the case $c=1$ and 
 $n\leq 14$ he proved that a general projection $\pi$ would be a stable map, and as a consequence
 he was able to show that, in this case, the fibers of $\pi$ have degree $\leq n+1$.
 More generally, in case
$n \leq 6c+7$, or $n \leq 6c+8$ and $c\leq 3$, he showed that the degree of any fiber of $\pi$ 
is  bounded by $n/c+1$. He also proved that for any $n$ and $c$, the number of distinct
points in any fiber is bounded by $n/c+1$; this is a special case of his result bounding the 
Thom-Boardman invariants.

An optimist (such as the second author), seeing these facts, 
might hope that the degree of the fibers of $\pi$ would be bounded by $n/c+1$ for any $n$ and $c$.
However, Lazarsfeld \cite{laz-book} (Volume 2, Prop. 7.2.17) showed that the singularities of $\pi(X)$ could have very high
multiplicity when $n$ is large. His ideas can also be used to prove that 
for large $n$
and a sufficiently positive embedding of any smooth variety $X$ in a projective space, 
a general linear projection
of $X$ to $\PP^{n+c}$ will have  
fibers of  degree exponentially greater than $n/c$. The first case with $c=1$ in which 
his argument gives something interesting is $n=56$, where it shows that (if the embedding
is sufficiently positive) then there will be fibers of degree $\geq 70$.
For a proof see \cite[Proposition 2.2]{projection}.

Although we know no upper bound on the degrees of the fibers of $\pi$  that depends only  on $n$ and $c$,
 we showed
in our paper \cite{projection} that there is a natural invariant of the fiber that agrees ``often'' with the degree
and that is always bounded by $n/c+1$.

In this note,  we 
generalize the construction there and give a general mechanism for producing 
such invariants.  Our approach is closely related to that of John Mather.

Here is a sample of the results we prove. We first ask how ``bad'' a finite scheme $Y$ can be
and still appear inside the fiber of the generic projection of a smooth n-fold to 
$\PP^{n+c}$? Our result is written in terms of the degree of $Y$ and the degree of
the \emph{tangent sheaf} to $Y$, defined as 
$$
\T_Y = Hom_{\cO_Y}(\Omega_{Y/K}, \cO_Y).
$$
We prove the following in Theorem \ref{one subscheme}.
\begin{theorem}\label{one scheme}
Let $X \subset \P^r$ be a smooth projective variety of dimension $n$, and let $Y$ be a scheme of dimension zero. 
If for a general linear projection $\pi_{\Sigma}: X \to \P^{n+c}$, there is a fiber of  $\pi_{\Sigma}$  
that contains $Y$ as a closed 
subscheme, then 
$$ 
  \deg Y +  \frac{1}{c} \deg \T_Y \leq \frac{n}{c}+1 .
$$   
\end{theorem}

This result easily implies (the special case for projections) of Mather's result bounding the
Thom-Boardman invariants (itself a special case of the transversality theorem he proves.)
This is because, as Mather shows, the Thom Boardman invariant of a germ at a point
 is determined by knowing whether certain subschemes are or are not contained in the fiber.
By way of example, we carry out the proof of the following useful special case:

\begin{corollary}[Mather] \label{thm1}
Let $X \subset \P^r$ be a smooth subvariety of dimension $n$, and let $\pi: X \to \P^{n+c}$ be a general projection with $c \geq 1$. 
Let $p$ be a point in $\P^{n+c}$, 
and assume $\pi^{-1}(p)$ consists of $r$ distinct points $q_1, \dots, q_r$. Denote by 
$d_i$ the corank of $\pi$ at $q_i$. Then we have 
$$ 
\sum_{ 1 \leq i \leq r} (\frac{d_i^2}{c}+d_i+1) \leq \frac{n}{c}+1.
$$ 
In particular, the number of distinct points in every fiber is bounded by $n/c+1$.
\end{corollary}

Mather's approach to this theorem works because the subschemes involved
in defining the Thom-Boardman singularities have no moduli---there is a discrete family of 
``test schemes". In other situations it is much more common for a certain ``type'' of subscheme to appear in a fiber,
although the subschemes themselves have non-trivial moduli. We can prove a result (Theorem \ref{family of schemes})
taking the dimension of the moduli space into account that sometimes gives sharper results.
Suppose, for example, that you know that  a generic projection from the smooth $n$-fold
$X$ to $\PP^{n+1}$ always has a fiber isomorphic to one of the schemes
$Y_F := \Spec k[x,y,z]/F+(x,y,z)^5$, where $F$ varies over all nonsingular cubic forms. This 
``truncated cone over an elliptic curve'', which has degree 31, varies with one parameter of moduli. The only
obvious subscheme common to all the $Y_F$ is 
$\Spec k[x,y,z]/(x,y,z)^3$. 
With 
Theorem \ref{one scheme} we get the bound
$n\geq 36$. But if we apply Theorem \ref{family of schemes} to the 1-dimensional moduli
family of $Y_F$,
we get the much stronger bound
$n \geq 69$.

One motivation for the study of the complexity of the fibers of general projections comes from 
the Eisenbud-Goto conjecture \cite{eg}, which states that the regularity of a projective subvariety of 
$\P^r$ is $\leq \deg(X)-\codim(X)+1$. An approach to this conjecture, which has been used to prove the conjecture for smooth surfaces and to prove a 
slightly weaker bound for smooth varieties of dimension at most 5 (see \cite{laz} and \cite{kwak}), is to bound the 
the regularity of the fibers of general projections. 

\begin{conjecture} Let $X \subset \P^r$ be a smooth projective variety of dimension $n$, and let 
$\pi:X \to \P^{n+c}$ be a general linear projection. If $Z \subset X$ is any fiber, then the 
Castelnuovo-Mumford regularity of $Z $ as a subscheme of $\P^r$ is at most $n/c+1$.
\end{conjecture}
 
The truth of the conjecture would imply that the Eisenbud-Goto conjecture holds up to a constant
 that depends only on $n$ and $r$ and is given explicitly in \cite{projection}. If true, the conjecture is sharp in some cases: The ``Reye Construction'' gives an Enriques
surface in $\PP^5$ whose projections to $\PP^3$ all have 3 colinear points in some fibers; and
an argument 
of Lazarsfeld shows that if $X$ is a Cohen-Macaulay variety of codimension 
2 in $\P^{n+2}$, and if $X$ is not  contained in a hypersurface of degree $\leq n$, then any projection 
of $X$ into $\P^{n+1}$ has fibers of length $n+1$. In this case any fiber is colinear. 
Since a scheme consisting of $n+1$ colinear points has regularity $n+1$, we 
get fibers of regularity $=n+1$  in these examples (see \cite{projection} for proofs.)

\section{Notation}
We will work over an algebraically closed field $k$ of characteristic zero.
If $T$ is a coherent sheaf of finite support on some scheme, we identify $T$ with its module of global
sections and write $\deg T$ for the vector space dimension of this module over $k$.

We fix $r \geq 2$, and we denote by $\G_k$ the Grassmannian of linear subvarieties of codimension 
$k$ in $\P^r$. 
Let $X$ be a smooth projective variety of dimension $n$, and let $c \geq 1$. A  
linear projection $X \to \P^{n+c}$ is determined by a sequence $l_1,\dots,l_{n+c+1}$
 of $n+c+1$ independent linear
forms on $\P^r$ that do not simultaneously vanish at any point of $X$. Associated to such a projection
is the projection center $\Sigma$, the linear space of codimension $n+c+1$ defined
by the vanishing of the $l_i$. The map taking a linear projection to the associated projection center
makes this set of projections into a $\PGL(n+c)$-bundle over $U\subset \G_{n+c+1}$
of planes $\Sigma$ that do not meet $X$.

We denote by $\pi_{\Sigma}$ the linear projection $X \to \P^{n+c}$ with center $\Sigma$. 
The morphism $\pi_{\Sigma}$ is birational, and its fibers are all zero-dimensional. 
The fibers of $\pi$ have the form $X\cap \Lambda$, where $\Lambda \in \G_{n+c}$
contains $\Sigma$. 

We will keep this notation throughout this paper.

\section{Measuring the Complexity of the Fibers}

Let $X \subset \P^r$ be a smooth subvariety of dimension $n$, 
and let $H$ be a subscheme of $\G_{n+c}$. 
For $[\Lambda] \in H$, set $Z = 
\Lambda \cap X$, and assume $\dim Z = 0$. Consider the restriction map 
$$
\rho: T_{\G_{n+c}, [\Lambda]} = H^0(N_{\Lambda/\P^r}) \to N_{\Lambda/\P^r}|_Z
$$
Let $V_G \subset N_{\Lambda/\P^r}|_Z$ be the image of $\rho$ and let $V_{H} = \rho(T_{H, [\Lambda]}) \subset V_G$.
Denote by $\O_Z V_{H} $ the $\O_Z$-submodule of $N_{\Lambda/\P^r}|_Z $ generated by $V_H$, and let $Q$ be the quotient module:
$$ 
0 \to \O_Z V_H \to N_{\Lambda/\P^r}|_Z \to Q \to 0.
$$

Here is our main technical result:

\begin{theorem}\label{invariant}
Let $X \subset \P^r$ be a smooth subvariety of dimension $n$, and let $H$ be 
a locally closed irreducible subvariety of $\G_{n+c}$, $c \geq 1$. Assume 
that for a general $[\Sigma]$ in $\G_{n+c+1}$, there is  
$[\Lambda] \in H$ such that $\Sigma \subset \Lambda$.  Then for a general $[\Lambda] \in H$, either $\Lambda \cap X$ is empty, or   
$$\deg Q  \leq n+c.$$   
\end{theorem}

The proof uses  the following result, which will also be used in the proof of
Theorem \ref{trans}
 
 \begin{lemma}\label{f}
Let $X$ be a smooth variety of dimension $n$ in $\P^r$, and let $H$ be a smooth locally closed 
subvariety of 
$\G_{n+c}$. Assume that for a general  $[\Sigma]$ in $\G_{n+c+1}$, there is $[\Lambda] \in H$ such that 
$\Sigma \subset \Lambda$. Let $[\Sigma]$ be a general point of $\G_{n+c+1}$ and let $[\Lambda]$ be a point of 
$H$ such that $\Sigma \subset \Lambda$. If $Q$ is as in Theorem \ref{invariant}, then the map 
$$
H^0(N_{\Lambda/\P^r} \otimes \O_{\Lambda}(-1)) \to Q
$$ 
is surjective. 
\end{lemma}

\begin{proof}
The restriction map $N_{\Lambda/\P^r} \to N_{\Lambda/\P^r} |_Z$ followed by the surjective map 
$N_{\Lambda/\P^r} |_Z \to Q$ gives a surjective map 
of $\O_{\Lambda}$-modules $N_{\Lambda/\P^r} \to Q$. We denote 
the kernel by $F$:
$$ 
0 \to F \to N_{\Lambda/\P^r} \to Q \to 0.
$$

We first show that the restriction map  $H^0(F) \to H^0(F|_\Sigma)$ is surjective. Consider the incidence correspondence 
$$
J =\{([\Sigma], [\Lambda]): \Sigma \subset \Lambda, [\Lambda] \in H \} 
\subset \G_{n+c+1} \times H,
$$  
and assume that $([\Sigma],[\Lambda])$ is  a general point of $J$. By our assumption the projection map $\pi_1: J \to \G_{n+c+1}$ is dominant. 
Since $H$ is smooth, and since the fibers of the projection map $\pi_2: J \to H$ are smooth, $J$ is smooth as well. Thus by generic smoothness, $\pi_1$ is smooth at  $([\Sigma], [\Lambda])$, and so  the map on Zariski tangent spaces $T_{J, ([\Sigma], [\Lambda])} \to T_{\G_{n+c+1}, [\Sigma]}$ is surjective. 

The short exact sequence of $\O_{\Sigma}$-modules 
 $$
  0 \to N_{\Sigma/\Lambda} \to N_{\Sigma/\P^r} \to 
 N_{\Lambda/\P^r}|_{\Sigma} \to 0
 $$
gives a surjective map $H^0(N_{\Sigma/\P^r}) \to H^0(N_{\Lambda/\P^r}|_{\Sigma})$. Note that since $\Sigma$ is general, $\Sigma \cap X = \emptyset$, and since $Q$ is supported on $\Lambda\cap X$, $F|_{\Sigma} 
= N_{\Lambda/\P^r}|_{\Sigma}$. It follows from the following commutative diagram
 $$
 \xymatrix{ T_{J, ([\Sigma], [\Lambda])} \ar@{>>}[r] \ar[d] & T_{\G_{n+c+1}, [\Sigma]} = 
 H^0(N_{\Sigma/\P^r}) \ar@{>>}[r] & H^0(N_{\Lambda/\P^r}|_{\Sigma}) 
 \ar[d]^{=} \\
 T_{H, [\Lambda]} \ar[r] & H^0(F) \ar[r] & H^0(F|_{\Sigma})
 }
 $$
then that $H^0(F) \to H^0(F|_\Sigma)$ is surjective.

Consider now the short exact sequence 
$$ 
 0 \to F \otimes \O_{\Lambda}(-1) \to F \to F|_{\Sigma} \to 0.
$$
Since $F|_{\Sigma} = N_{\Lambda/\P^r}|_{\Sigma}$, we have 
$H^1(F|_{\Sigma}) = H^1(N_{\Lambda/\P^r}|_{\Sigma})= 0$. Since  $H^0(F) \to H^0(F|_\Sigma)$ is surjective, 
we get $H^1(F \otimes \O_{\Lambda}(-1)) \cong  H^1(F)$.
Therefore, the image of the map $H^0(N_{\Lambda/\P^r}) \to Q$ is the same as the image of 
the map $H^0(N_{\Lambda/\P^r}\otimes \O_{\Lambda} (-1)) \to Q$, and thus  
by Proposition \ref{dim0}
both of these maps are surjective. 
\end{proof}

\begin{proposition}{\cite[Proposition 3.1]{projection}}\label{dim0} 
Suppose that $\delta:A\to B$ is an epimorphism of 
coherent sheaves on $\PP^r$,
and suppose that $A$ is generated by global sections.
If $\delta(H^0(A)) \subset H^0(B)$ has the same dimension
as  $\delta(H^0(A(1))) \subset H^0(B(1))$, then $\dim B = 0$ and
$\delta(H^0(A(m)))=H^0(B(m))\cong H^0(B)$ for all $m\geq 0$.\qed
\end{proposition}

\begin{proof}[Proof of Theorem \ref{invariant}]
Assume that for a general $[\Lambda]$ in $H$, $\Lambda\cap X$ is non-empty. It follows from Lemma \ref{f}, applied to the smooth locus in $H$,
 that the map $H^0(N_{\Lambda/\P^r} 
\otimes \O_{\Lambda}(-1)) \to Q$ is surjective. Therefore, 
$$
\deg Q  \leq \dim H^0(N_{\Lambda/\P^r} 
\otimes \O_{\Lambda}(-1)) = \dim H^0(\O_{\Lambda}^{n+c}) = n+c.
$$
\end{proof}

Since we have 
$$
(n+c) \deg Z = \deg N_{\Lambda/\P^r}|_Z = \deg \O_Z V_H + \deg Q,
$$
where $Z = \Lambda\cap X$, it follows from the above theorem that
any upper bound on the degree of $\O_Z V_H$ puts some restrictions on the fibers of $\pi_{\Sigma}$ for general $\Sigma$. 
 
\begin{corollary}\label{fiber-bound}
Let $X \subset \P^r$ be a smooth projective variety of dimension $n$, and let $c \geq 1$ be an integer. Let $H$ be a locally closed 
irreducible subvariety of $\G_{n+c}$, and assume that for a general projection $\pi_{\Sigma}$, there is 
$[\Lambda] \in H$ that contains $\Sigma$. Then for a general $[\Lambda] \in H$    
$$
\deg Z \leq \frac{\deg \O_Z V_H}{n+c} + 1.
$$
where $Z=\Lambda\cap X$.
\end{corollary}

For example, if we apply Theorem \ref{invariant} to the scheme $H\subset \G_{n+c}$ whose points 
correspond to planes that
intersect $X$ in schemes of length $\geq l$, for some integer $l\geq 1$, we recover
the central result  of our paper \cite{projection}. recall that for varieties $X,Y \subset \PP^r$
that meet in a scheme $Z=X\cap Y$ of dimension zero, with 
$\codim Y -\dim X>0$ we there  defined
$q(X,Y)$ to be
$$
q(X,Y) = \frac{ \deg \coker \Hom(\cI_{Z/X}/\cI_{Z/X}^2, \cO_Z) \to \Hom(\cI_{Y/P}/\cI_{Y/P}^2, \cO_Z)}
{\codim Y -\dim X}.
$$
For example, if $X, Y$ are smooth and $Z$ is a locally complete intersection scheme then $q(X,Y) = \deg X\cap Y$,
and more generally $q$ is a measure of the difficulty of flatly deforming $Y$ in such a way that $Z=X\cap Y$ 
deforms flatly as well.
\begin{theorem}\label{main1}
If $X$ is a smooth projective variety of dimension $n$ in $\P^r$, 
and if $\pi: X \to \P^{n+c}$ is a general projection, then 
every fiber $X\cap \Lambda$, where $\Lambda$ is a linear
subspace containing the projection center in codimension 1, satisfies:
$$
q(X,\Lambda)\leq \frac{n}{c}+1.
$$
\end{theorem}

In \cite{projection} we derived explicit bounds on the lengths of fibers from this result.

\subsection{ A Problem}
Fix positive integers $l$ and $m$, and let $H$ be the reduced subscheme of $\G_{n+c}$ 
consisting of those planes $\Lambda$ such that $\deg \Lambda \cap X = l$ and $
\deg \Omega_{\Lambda \cap X} = m$.  Assume that for a general projection $\pi: 
\Sigma \to \P^{n+c}$, there is a fiber $\Lambda \cap X$ such that $[\Lambda]$ is in $H$. 
We would like to use Corollary \ref{fiber-bound} in this case to get a bound 
on the fibers of general projections stronger than,
say,  that of 
Corollary 1.2 in \cite{projection}. 

Assume that $[\Lambda] \in H$ is a general point, and let 
$V_H$ be the image of $T_{H, [\Lambda]}$ in $N_{\Lambda/\P^r}|_Z$. Then 
since we assume that the length of the intersection with $X$ is fixed for all points of $H$, $V_H$ is a subspace of $N_{Z/X}$.  If we denote by $V' \subset N_{Z/X}$ the tangent space to the space of first order deformations of $Z$ in $X$ that keep the degree of $\Omega_Z$ fixed, then 
$V_H \subset V'$. 

If $Z$ is an arbitrary zero-dimensional subscheme of a smooth 
variety $X$, then 
$V'$ is not necessarily a submodule of $N_{Z/X}$. For example, if 
we let $Z$ be the subscheme of $\A^2$ defined by the ideal $I=<x^4+y^4, xy(x-y)(x+y)(x-2y)>$, then 
$Z$ is supported at the origin and is of degree 20. Using Macaulay 2, we find that the space of deformations of $Z$ in $\A^2$ 
that fix the degree of $\Omega_Z$ is a vector space of dimension 17, but the $\O_Z$-module generated by this space is a vector space of dimension 18.  

Is there an upper bound on the dimension of the submodule generated by $V_H$ that 
is stronger than $\deg N_{Z/X}$? In the special case, when $Z$ is curvilinear of degree 
$m$ (so $\deg \Omega_Z = m-1$), $V'$ is a submodule of degree = $ m \dim X - (m-1) < \deg 
N_{Z/X}$. But if $Z$ is of arbitrary degree and not curvilinear, we currently know no bound that could improve Corollary 1.2 in \cite{projection}.

\section{General Projections Whose Fibers Contain Given Subschemes}  
Fix a zero-dimensional scheme $Y$. We wish to give a bound on invariants
of $Y$ that must hold if $Y$ appears as a subscheme of a fiber of the
general projection of $X$ to $\PP^{n+c}$. Our result generalizes 
a key part of the proof of Mather's theorem bounding the Thom-Boardman
invariants of a general projection.

\def\T{{\mathcal T}}
In the following we write $\T_Y$ for the tangent sheaf $\T_Y := Hom_{\O_Y}(\Omega_Y, \O_Y)$ of $Y$, and 
similarly for $X$. 
\begin{theorem}\label{one subscheme}
Let $X \subset \P^r$ be a smooth projective variety of dimension $n$, and let $Y$ be a scheme of dimension zero. 
If for a general linear projection $\pi_{\Sigma}: X \to \P^{n+c}$, there is a fiber of  $\pi_{\Sigma}$  
that contains $Y$ as a closed 
subscheme, then 
$$ 
  \deg Y +  \frac{1}{c} \deg \T_Y \leq \frac{n}{c}+1 .
$$   

\end{theorem}

\begin{proof}
Let  $\Hom(Y,X)$ be the space of morphisms from $Y$ to $X$, and 
let $I \subset \Hom(Y, X) \times \G_{n+c} $ be the incidence correspondence parametrizing 
the pairs $([i], [\Lambda])$ such that $i$ is a closed immersion from $Y$ to $\Lambda\cap X$. 
Denote by $H$ the image of $I$ under the projection map $\Hom(Y,X) \times 
\G_{n+c}\to \G_{n+c}$. 
We give $H$ the reduced induced scheme structure. 

Let $([i], [\Lambda])$ be a general point of $I$, and set $Z:= \Lambda \cap X$. 
We consider $Y$ as a closed subscheme of $X$, and we let 
$N_{Y/X} =\Hom(I_{Y/X},\O_Y)$ denote the normal sheaf of $Y$ in $X$. 
The Zariski tangent space to $\Hom(Y,X)$ at $[i]$ is isomorphic to $H^0(\T_X|_Y)$ 
(\cite[Theorem I.2.16]{kollar}). 

Denote by $M'$ the $\O_Y$-submodule of $N_{\Lambda/\P^r}|_Y$ generated by the image of the 
restriction map from the Zariski tangent space:
$$
\xymatrix{
\rho_Y: T_{H,[\Lambda]} \ar@{^{(}->}[r] & H^0(N_{\Lambda/\P^r}) \ar[r] & N_{\Lambda/\P^r}|_Y.
}
$$

We first claim that $\deg M' \geq (n+c) \deg Y - (n+c)$. Let $Q'$ be the quotient of $N_{\Lambda/\P^r}|_Y$ by $M'$:
$$
0 \to M' \to N_{\Lambda/\P^r}|_Y \to Q' \to 0.
$$
Denote by $M$ the submodule of $N_{\Lambda/\P^r}|_Z$ 
generated by the image of 
$$
\xymatrix{
\rho_Z: T_{H,[\Lambda]} \ar@{^{(}->}[r] & H^0(N_{\Lambda/\P^r}) \ar[r] & N_{\Lambda/\P^r}|_Z,
}
$$
and let $Q$ be the cokernel:
$$ 
0 \to  M \to N_{\Lambda/\P^r}|_Z \to Q \to 0.
$$
The surjective map $N_{\Lambda/\P^r}|_Z \to N_{\Lambda/\P^r}|_Y$
carries  $M$ into $M'$, and thus induces a surjective map 
$Q \to Q'$. Since by Theorem \ref{invariant}, 
$\deg Q \leq n+c$, we have
$\deg Q' \leq n+c$, and so 
$$
\deg M' = \deg N_{\Lambda/\P^r}|_Y -\deg Q'  \geq (n+c) \deg Y - (n+c).
$$
This establishes the desired lower bound.

We next give an upper bound on $\deg M'$. Since $X$ is smooth, dualizing the surjective map $\Omega_X|_Y \to \Omega_Y$ into $\O_Y$, 
we get an injective map $\T_Y \to \T_X|_Y$ and an exact sequence 
$$
\xymatrix{
0 \ar[r] &\T_Y \ar[r] & \T_X|_Y \ar[r]^{\phi} & N_{Y/X}.
}
$$
Let $\pi_1$ and $\pi_2$ denote the projections maps from $I$ to $\Hom(X,Y)$ and $H$ respectively. 
We have a diagram 
$$
\xymatrix{ T_{I, ([i], [\Lambda])}  \ar@{>>}[r]^{d\pi_2} \ar[d]^{d\pi_1} & T_{H, [\Lambda]}  
\ar@{^{(}->}[r]& 
T_{\G_{n+c}, [\Lambda]}=H^0(N_{\Lambda/\P^r}) \ar[dd]^{\rho_Y}  \\
T_{\Hom(X,Y), [i]}=\T_X|_Y \ar[d]^{\phi} & & \\
N_{Y/X} \ar[rr]^{\psi} & &N_{\Lambda/\P^r}|_Y
}
$$
where $\psi$ is obtained by dualizing the 
map $I_{\Lambda/\P^r} \otimes \O_X \to I_{Y/X}$ into $\O_Y$. 

It follows from the diagram that 
$\rho_Y(T_{H,[\Lambda]})$ is contained in the image of $\psi\circ \phi$. Since the image of $\psi\circ \phi: \T_X|_Y \to N_{\Lambda/\P^r}|_Y$ 
is a submodule of $N_{\Lambda/\P^r}|_Y$, $M'$ is contained in the 
image of $\psi\circ \phi$ as well. Thus the degree of $M'$ is less than or equal to the degree of the image 
of $\phi$, that is   
$$
\deg M' \leq \deg \T_X|_Y - \deg \T_Y = n \deg Y - \deg \T_Y.
$$
Comparing this upper bound with the lower bound we got earlier completes the proof.
\end{proof}

Theorem \ref{one subscheme} was inspired by the results of Mather on the Thom-Boardman invariants.
Mather shows that the Thom-Boardman symbol of a germ of a map is determined by which of
a certain discrete set of different subschemes the fiber contains. These are the schemes of the form
$\Spec k[x_1,\dots,x_n]/(x_1)^{t_1} + (x_1,x_2) ^{t_2}+\cdots+(x_1,\dots,x_n)^{t_n}$. 
His result that the general projection is transverse to the Thom-Boardman strata is closely
related (see also Section \ref{transversality section}.) We illustrate by proving the special case 
announced in the introduction.

\begin{proof}[Proof of Corollary \ref{thm1}]
For $d \geq 1$, let $A_d = \frac{k[x_1, \dots, x_d]}{m^2}$ where $m$ is the ideal generated by 
$x_1, \dots, x_d$. If $q \in X$ is a point of corank $d$ for the projection $\pi$, then there is 
a surjective map $\O_{\pi^{-1}(\pi(q)), q} \to A_d$. 

Fix an integer $r \geq 1$, and fix a sequence of coranks $d_1 \geq \dots \geq d_r \geq 0$. If we  
denote by $Y$ the disjoint union of the schemes $\spec A_{d_i}, 1 \leq i \leq r$, then we have 
$\deg Y = \sum_{1\leq i \leq r} (d_i+1)$ and $\deg \T_Y  = \sum_{1 \leq i \leq r} d_i^2$. 

Assume now that 
for a general linear projection $\pi: X \to \P^{n+c}$, there is a fiber consisting of at least 
$r$ points $q_1, \dots, q_r$ such that the corank of $\pi$ at $q_i$ is at least $d_i$ for 
$1 \leq i \leq r$.  
Then for a general linear projection with center $\Sigma$, there is a fiber $X \cap \Lambda$ and 
a closed immersion $i: Y \to \Lambda \cap X$. It follows from the previous 
theorem that 
$$ 
\sum_{1 \leq i \leq r}(\frac{d_i^2}{c} + d_i+1) = \deg Y + \frac{1}{c} \deg \T_Y \leq \frac{n}{c}+1. 
$$
\end{proof}

Except in a few situations, such as the Thom-Boardman computation above, it is more likely
that for a general projection the fiber might be ``of a certain type'', or contain one of a given
family of special subschemes. 
In the next theorem, we generalize Theorem \ref{one subscheme} to 
such a family of zero-dimensional schemes. We have separated the proofs because this
version involves considerably more technique. But we do not repeat the final part of the argument,
since it is the same as before.

\begin{theorem}\label{family of schemes}
Let $X \subset \P^r$ be a smooth projective variety of dimension $n$.  Suppose that $B$ is an integral 
scheme of dimension $m$ and that $p: U \to B$ is a flat family of zero dimensional schemes 
over $B$. For a point $b \in B$, let 
$\T_{U_b} = \Hom(\Omega_{U_b}, \O_{U_b})$.
If for a general projection $\pi_{\Sigma}: X \to \P^{n+c}$, $c \geq 1$, there is a fiber 
$U_b$ of $p:U \to B$ such that $U_b$ can be embedded in one of the fibers of $\pi_{\Sigma}$, then 
$$ 
(1-\frac{m}{c})  \deg U_b +\min_{b \in B} \{  \frac{1}{c}  \deg \T_{U_b}\}  \leq \frac{n}{c}+1. 
 $$
\end{theorem}

\begin{proof}
Passing to a desingularization, we can assume that $B$ is smooth. Let $Hom_B(U, B \times X)$ be the functor 
 $$
Hom_B(U, X \times B)(S) = \{ B\mbox{-morphisms}: U \times_B S \to (X \times B) \times_B S \}.
$$
By \cite[I.1.10]{kollar}, this functor is represented by a scheme $\Hom_B(U, X\times B)$ over $B$ that is isomorphic to an open subscheme of 
$\Hilb( U \times X/B)$. The closed points of $\Hom_B(U, X\times B)$ parametrize morphisms from 
fibers of $p:U \to B$ to $X$.

Denote by $I \subset \Hom_B(U, X \times B) \times \G_{n+c}$ 
the incidence correspondence consisting of 
the points $([i], [\Lambda])$ such that $i$ is a closed immersion to 
$\Lambda \cap X$, and let $H$ be the image of the projection map $I \to \G_{n+c}$. We give 
$H$ the reduced structure as a subscheme of  $\G_{n+c}$. 

Assume that $([i], [\Lambda]) $ is a general point of $I$, and assume that $[i]$ represents 
the closed immersion $i: U_b \to \Lambda \cap X$, $b \in B$. Set $Y =U_b$ and $Z = \Lambda \cap X$.  
Denote by $M'$ the $\O_Y$-submodule of $N_{\Lambda/\P^r}|_Y$ generated by the image of the 
restriction map 
$$
\xymatrix{
\rho_Y: T_{H,[\Lambda]} \ar@{^{(}->}[r] & H^0(N_{\Lambda/\P^r}) \ar[r] & N_{\Lambda/\P^r}|_Y.
}
$$

As in the proof of Theorem \ref{one subscheme}
we need an upper bound on the degree of $M'$. Since $i:Y \to X$ is a closed immersion, there is a natural map on the Zariski tangent spaces:
$$
\phi: T_{\Hom_B(U, X \times B) ,[i]} \to T_{\rm Hilb(X),[Y]} = N_{Y/X}.$$
(If $V$ is a flat family over $D := \Spec k[\epsilon]/\epsilon ^2$, and if $f:V \to X \times D$ is such that 
$f_0: V_0 \to X$ is a closed immersion, then so is $f$. Thus a morphism
$D \to \Hom_B(U, X \times B)$ gives a natural morphism $D \to \rm Hilb(X)$.)
As a first step toward bounding the degree of $M'$ we will show
 that the $\O_Y$-submodule of $N_Y$ generated by the image of $\phi$ 
has degree at most $(n+m) \deg Y - \dim \T_Y.$

Note that the fiber of the map $\Hom_B(U, X \times B) \to B$ over $b$ is 
$\Hom(Y, X)$. Therefore, the vertical Zariski tangent space to $\Hom_B(U, X \times B) $ at $[i]$ is isomorphic to 
$H^0(\T_X|_Y)$. Since $Y$ is zero-dimensional, we may identify $H^0(\T_X|_Y)$ with $\T_X|_Y$.

Let $(Q,\rm m_Q)$ be the local ring of $\Hom_B(U, X \times B)$ at $[i]$, and let $\rm m_b$ be the 
maximal ideal of the local ring of $B$ at $b$. There is an exact sequence of $k$-vector spaces: 
$$
(\rm m_bQ+\rm m_Q^2)/m_Q^2 \to \rm m_Q/\rm m_Q^2 \to \rm m_Q/(\rm m_bQ+\rm m_Q^2) \to 0.
$$
Since $(\rm m_Q/(\rm m_bQ+\rm m_Q^2))^*$ is the vertical Zariski tangent space at $[i]$, it is equal to 
$H^0(\T_X|_Y)$. So dualizing the above sequence, we get an exact sequence 
$$ 
0 \to H^0(\T_X|_Y) \to T_{\Hom_B(U, X \times B) ,[i]} \to V.
$$
Since $B$ is smooth of dimension $m$ we see that
 $V = Hom((\rm m_bQ+\rm m_Q^2)/m_Q^2, Q/{\rm m}_b)$ is a vector space of dimension is $\leq m$.  

Denote by $N$ the quotient of $\T_X|_Y$ by $\T_Y$, and consider the diagram
$$
\xymatrix{
&\T_Y \ar@{^{(}->}[d]\\
0 \ar[r] &\T_X|_Y \ar[r] \ar@{>>}[d] & T_{\Hom_B(U, X \times B) ,[i]} \ar[r] \ar[d]^{\phi} & V\\
& N \ar@{^{(}->}[r] & N_{Y/X} \\
}
$$
The image of $\T_{X}|_Y$ under $\phi$ is contained in the image of $N \to N_{Y/X}$, which is 
a $\O_Y$-submodule of degree $\leq \deg \T_{X}|_Y - \deg \T_Y = n \deg Y - \deg \T_Y.$
The degree of the $\O_Y$-submodule generated by the image of $\phi$ in $N_{Y/X}$ is therefore 
$$
\leq n \deg Y - \deg \T_Y + \dim V  \cdot \deg Y = (n+\dim B) \deg Y - \dim \T_Y.
$$
This establishes the desired upper bound on the degree of the
$\O_Y$-submodule of $N_Y$ generated by the image of $\phi$.

Let $\pi_1$ and $\pi_2$ denote the projections maps from $I$ to $\Hom_B(U, B \times X)$ and 
$H$ respectively. We have a diagram 
$$
\xymatrix{ T_{I, ([i], [\Lambda])}  \ar@{>>}[r]^{d\pi_2} \ar[d]^{d\pi_1} & T_{H, [\Lambda]}  
\ar@{^{(}->}[r]& 
T_{\G_{n+c}, [\Lambda]}=H^0(N_{\Lambda/\P^r}) \ar[dd]^{\rho_Y}  \\
T_{\Hom_B(U, B \times X), [i]} \ar[d]^{\phi} & & \\
N_{Y/X} \ar[rr]^{\psi}  & &N_{\Lambda/\P^r}|_Y
}
$$
The diagram shows that $M'$ is contained in the submodule generated by the image of the 
composition of the maps $T_{\Hom_B(U, B \times X), [i]} \to N_{Y/X} \to N_{\Lambda/\P^r}|_Y$, so 
by the upper bound established before we have 
$$ 
\deg M' \leq (n+m) \deg Y - \dim \T_Y.
$$
On the other hand, the same argument as in the proof of Theorem \ref{one subscheme} shows that 
$$
\deg M' \geq (n+c) \deg Y - (n+c).
$$
Thus, we get   
$$
 (c - m) \deg Y + \deg \T_Y \leq n+c.
 $$
 Since $\deg \T_{U_b}$ is an upper semicontinuous function on $B$, we get the desired result.    
 \medskip

\end{proof}



\section{A Transversality Theorem}
\label{transversality section}

Mather defines a property of a smooth subvariety of a jet bundle that he calls \emph{modularity},
and proves that a general projection has jets that are transverse to any modular subvariety.
We will make a related, but different definition of a more global sort, and prove a 
transversality result that can be used to give a different derivation of some of Mather's results.

Let $H$ be a locally closed subscheme of $G :=\G_{n+c}$. Let $[\Lambda] \in H$ be such that  
$ Z := \Lambda \cap X$ is zero-dimensional. Consider the restriction map 
$$ 
\rho: T_{G, [\Lambda]} = H^0(N_{\Lambda/\P^r}) \to N_{\Lambda/\P^r}|_Z,
$$ and set $V_G = \rho(T_{G,[\Lambda]})$ and $V_H= \rho(T_{H, [\Lambda]}) \subset V_G$. Denote by 
$\O_ZV_H$ the $\O_Z$-submodule of $N_{\Lambda/\P^r}|_Z$  generated by $V_H$. 

We call $H$ {\em \modular with respect to $X$} if whenever $Z = \Lambda \cap X$ is zero-dimensional, we have 
$$
V_G \cap \O_ZV_{H} = V_{H}.
$$

\begin{example}\label{modular-example}
We describe examples of \modular  and non-\modular  subschemes of $G =\G_{n+c}$. 
First, let $l\geq 0$ be an integer.  Let $U \subset \G_{n+c} \times \P^r$ be the 
universal family over $\G_{n+c}$, and let 
$$
U_X = U \cap (\G_{n+c} \times X) \subset \G_{n+c} \times \P^r 
$$ 
be the scheme theoretic intersection.  By \cite[I, Theorem 1.6]{kollar}, there is a locally closed subscheme $H$ 
of $\G_{n+c}$ with the following property: For any morphism $S \to \G_{n+c}$, 
the pullback of $U_X$ to $S$ is flat of relative dimension zero and relative degree $l$ if and only if 
$S \to \G_{n+c}$ factors through $H$. The closed points of $H$ parametrize those linear subvarieties whose degree of intersection with $X$ is $l$. 

 Denote the normal sheaf of $Z$ in $X$ by $N_{Z/X} = \Hom (I_{Z/X}, O_Z)$. 
 If $ I$ and $J$ are the ideal sheaves of $X$ and $\Lambda$ in $\P^r$ 
 respectively, 
 then there is a surjective map of $\O_Z$-modules
 $$
\frac{J}{J^2+ I J} 
\to \frac{J+ I}{J^2+I}.
$$
 Dualizing the above map into $\O_Z$, we get an injective map 
 $$ 
 N_{Z/X} \to N_{\Lambda/\P^r}|_Z.
 $$  If $M$ is the image of this map, then $M$ is an $\O_Z$-module, and  $V_G \cap M= V_H$,
 as one sees by considering morphisms from $\Spec k[t]/(t^2)$. Therefore $H$ is a \modular with respect to $X$.
\smallskip
  
  For the next example, fix a point $p \in X$, and let $q \not\in X $ be a point on a tangent line to $X$ at $p$. Let $H$ be the subvariety of $G$ that consists of those linear subvarieties of codimension 
  $n+c$ that pass through $q$. We claim that $H$ is not a \modular subvariety. Pick $[\Lambda] \in H$   
  so that it passes through $p$ and $q$, and set $Z= \Lambda \cap X$. 
  Choose a system of homogenous coordinates $x_0, \dots, x_{r-n-c}$ for $\Lambda$ such that 
  $p =(1: 0:
   \dots :0)$, $q = (0:1:0:\dots :0)$, and $Z \subset U:=\{x\in \Lambda \mid x_0 \neq 0 \}$. We have 
   $\O_U(U) = k[x_1, \dots, x_{r-n-c}]$, and since $q$ is on the tangent plane to $X$, for any 
   linear polynomial that vanishes on $Z$, the coefficient of $x_1$ is zero.  
   
   If we identify $T_{G, [\Lambda]}$ with the global sections of $N_{\Lambda/\P^r} \simeq \O_{\Lambda}(1)^{n+c}$, then  $T_{H, [\Lambda]}$ is identified with the $(n+c)$-tuples of linear forms $(L_1, \dots, L_{n+c})$ in $x_0, \dots, x_{r-n-c}$ vanishing at $q$. The image of 
  $$
  \rho: H^0(N_{\Lambda/\P^r}) \to N_{\Lambda/\P^r}|_Z \simeq \O_Z^{n+c}
  $$
  contains $(1,0,\dots,0)$ and $(x_1|_Z,0,\dots,0)$. 
  But $(1,0,\dots,0) \in 
  \rho(T_{H, [\Lambda]})$ and $(x_1|_Z,0,\dots,0) \not\in \rho(T_{H, [\Lambda]})$. Thus $H$ is not 
  semi-modular.   
\end{example}
  
We now turn to the transversality result.
If $f: Y_1 \to Y_2$ is a regular morphism between smooth varieties, and if $H$ is a smooth subvariety of 
$Y_2$, then $f$ is called  {\em transverse} to $H$ if for every $y$ in $Y_1$, either 
$f(y) \not\in H$ or  $$ 
 T_{H, f(y)} + df(T_{Y_1, y} ) = T_{Y_2, f(y)}
 $$
where $df: T_{Y_1, y} \to 
T_{Y_2, f(y)}$ is the map induced by $f$ on the Zariski tangent spaces.

\begin{theorem}\label{trans}
Let $X \subset \P^{r}$ be a smooth projective variety of dimension $n$, fix $c \geq 1$, and let $H$ be 
a smooth subvariety of $\G_{n+c}$ that is \modular with respect to $X$. 
For a general linear projection $\pi_{\Sigma}:X \to \P^{n+c}$, the map 
$$
\phi_{\Sigma} : \P^{n+c} \to \G_{n+c}
$$
that sends $y \in \P^{n+c}$ to the corresponding linear subvariety in $\P^r$ is transverse to $H$.
\end{theorem}


%

\begin{proof}
Let $[\Sigma]$ be a general point of $\G_{n+c+1}$. For $y \in \P^{n+c}$, let
$\Lambda \subset \P^r$ be the corresponding linear subvariety,
the preimage of $y$ under the projection map from $\PP^{n+c}$,
and set $Z =\Lambda \cap X$. 

Assume $[\Lambda] \in H$, and let $V_H$ be the image of  $T_{H, [\Lambda]}$ under the restriction map 
$$ 
T_{G, [\Lambda]} = H^0(N_{\Lambda/\P^r}) \to N_{\Lambda/\P^r}|_Z.
$$ 
Denote by $Q$ the quotient of  $N_{\Lambda/\P^r}|_Z$ by $\O_ZV_H$:
$$ 
0 \to \O_ZV_H \to N_{\Lambda/\P^r}|_Z \to Q \to 0.
$$
Then we can consider $Q$ as a sheaf of $\O_{\Lambda}$-modules that is supported on $Z$. 
Let $F = \ker (N_{\Lambda/\PP^r} \to Q)$, so that $\rho(T_{H,\Lambda}) \subset H^0(F).$
Since $H$ is \modular, $\O_ZV_H \cap V_G= V_H$, and hence 
$T_{H, [\Lambda]} = H^0(F)$. 

To prove the statement, note that if for a general $\Sigma$, there is 
no $y \in \P^{n+c}$ with $\phi([\Sigma], y) \in H$, then there is nothing to prove. 
Otherwise, by Lemma \ref{f}, the map 
$H^0(N_{\Lambda/\P^r} \otimes \O_{\Lambda}(-1)) \to Q$ is surjective. Hence 
If we consider $H^0(F)$ and $H^0(N_{\Lambda/\P^r} \otimes \O_{\Lambda}(-1))$ as 
subspaces of $H^0(N_{\Lambda/\P^r})$, then we get 
$$
H^0(F) + H^0(N_{\Lambda/\P^r} \otimes \O_{\Lambda}(-1)) = 
H^0(N_{\Lambda/\P^r}).
$$
If we identify $T_{\G_{n+c}, [\Lambda]}$ with the 
space of global sections of $N_{\Lambda/\P^r}$, then 
$d\phi_{\Sigma}(T_{\P^{n+c}, y})$ is identified with $H^0(N_{\Lambda/\P^r} \otimes 
\O_{\Lambda}(-1))$, and thus
$$
T_{H, \phi_{\Sigma}(y)} + d\phi_{\Sigma}(T_{\P^{n+c}, y} ) = H^0(N_{\Lambda/\P^r}) =T_{\G_{n+c}, 
\phi_{\Sigma}(y)}.
$$

\end{proof}

\bigskip
\newcommand{\closer}{\vspace{-1.5ex}}

\bigskip
\vbox{\noindent Author Addresses:\par
\smallskip
\noindent{Roya Beheshti}\par
\noindent{Department of Mathematics, Washington University, St. Louis,
MO 63130}\par
\noindent{beheshti@math.wustl.edu}\par
\smallskip
\noindent{David Eisenbud}\par
\noindent{Department of Mathematics, University of California, Berkeley,
Berkeley CA 94720}\par
\noindent{eisenbud@math.berkeley.edu}\par
}


\begin{thebibliography}{99}

{\small
\bibitem{projection}R.~Beheshti, D.~Eisenbud, Fibers of generic  projections, to appear in 
Compositio Math.

\bibitem{eg} D.~Eisenbud and S.~Goto, Linear free resolutions and minimal multiplicity, 
J. Algebra  88  (1984),  no. 1, 89--133. 

\bibitem{glp} L.~Gruson, R.~Lazarsfeld, and C.~Peskine, On a theorem of Castelnuovo, and the equations defining space curves, Invent. Math.  72  (1983),  no. 3, 491--506.

\bibitem{kollar} J.~Koll\'ar,  Rational curves on algebraic varieties, volume 32 of Ergebnisse der Mathematik und ihrer Grenzgebiete, 3. Folge. Springer-Verlag, 1996.

\bibitem{kwak} S.~Kwak, 
Generic projections, the equations defining projective varieties 
and Castelnuovo regularity, Math. Z. 234(2000),  no. 3, 413--43.

\bibitem{laz} R.~Lazarsfeld, A sharp Castelnuovo bound for smooth surfaces, 
Duke Math. J.  55 (1987),  no 2, 423--429.

\bibitem{laz-book} R. Lazarsfeld, {Positivity in Algebraic Geometry I.}
Ergebnisse der Math. Und ihrer Grenzgebiete 48. Springer-Verlag, Berlin,
2004.

\bibitem{mat} J.~N.~Mather,  
Stable map-germs and algebraic geometry, 
Manifolds-Amsterdam, Lecture Notes in Math No. 197, 
Springer-Verlag, 1971, pp. 176--193.

\bibitem{mather2} J.~N.~Mather, On Thom-Boardman singularities.  Dynamical systems (Proc. Sympos., Univ. Bahia, Salvador, 1971),  pp. 233--248. Academic Press, New York, 1973. 

\bibitem{mather} J.~N.~Mather,  Generic projections, Ann. of Math. (2)  98  (1973), 226--245.
}
\end{thebibliography}
 \end{document}